\documentclass[sort&compress]{elsarticle}

\usepackage{setspace}
\usepackage{amsmath}
\usepackage{amssymb}
\usepackage{mathrsfs}
\usepackage{multirow}
\usepackage{color}
\usepackage{longtable}
\usepackage{array}
\usepackage{url}
\usepackage{comment}
\usepackage{enumerate}
\usepackage{eucal}
\usepackage{graphics}
\usepackage{overpic}
\usepackage{rotating}

\usepackage{algorithm}
\usepackage[noend]{algorithmic}
\usepackage[table]{xcolor}

\usepackage{stmaryrd}

\usepackage{mathtools}
\usepackage{xparse}
\DeclarePairedDelimiterX{\cond}[1]{[}{]}{\setargs{#1}}
\DeclarePairedDelimiterX{\set}[1]{\{}{\}}{\setargs{#1}}
\DeclarePairedDelimiterX{\sparenv}[1]{\lbrack}{\rbrack}{\setargs{#1}}

\NewDocumentCommand{\setargs}{>{\SplitArgument{1}{;}}m}
{\setargsaux#1}
\NewDocumentCommand{\setargsaux}{mm}
{\IfNoValueTF{#2}{#1} {#1\,\delimsize|\,\mathopen{}#2}}

\DeclarePairedDelimiter\abs{\lvert}{\rvert}
\DeclarePairedDelimiter\ceil{\lceil}{\rceil}
\DeclarePairedDelimiter\floor{\lfloor}{\rfloor}
\DeclarePairedDelimiter\parenv{\lparen}{\rparen}

\newcommand{\cA}{\mathcal{A}}

\newcommand{\cC}{\mathcal{C}}

\newcommand{\cQ}{\mathcal{Q}}

\newcommand{\oc}{\overline{c}}

\newcommand{\bc}{\mathbf{c}}

\newcommand{\bv}{\mathbf{v}}

\newcommand{\ov}{\overline{v}}
\newcommand{\bu}{\mathbf{u}}
\newcommand{\ou}{\overline{u}}

\newcommand{\ow}{\overline{w}}

\newcommand{\ozero}{\overline{0}}

\renewcommand{\leq}{\leqslant}

\renewcommand{\geq}{\geqslant}

\newdefinition{definition}{Definition}
\newdefinition{remark}{Remark}
\newdefinition{question}{Problem}
\newtheorem{theorem}{Theorem}

\newtheorem{proposition}{Proposition}
\newtheorem{corollary}{Corollary}
\newtheorem{lemma}{Lemma}
\newproof{proof}{Proof}

\newcommand{\F}{\mathbb{F}}

\newcommand{\N}{\mathbb{N}}
\newcommand{\Prob}{\mathbb{P}}

\newcommand{\Zero}{{\mathbf{0}}}

\DeclareMathOperator{\wt}{wt}

\DeclareMathOperator{\supp}{supp}

\DeclareMathOperator{\E}{E}
\newcommand{\ind}{\mathbb{I}}
\newcommand{\tZ}{\widetilde{Z}}

\newcommand{\eqdef}{\triangleq}

\newenvironment{bsmatrix}{\left[\begin{smallmatrix}}{\end{smallmatrix}\right]}

\newcommand{\szu}{\mathsf{ZU}}
\newcommand{\szc}{\mathsf{ZC}}
\newcommand{\snu}{\mathsf{NU}}
\newcommand{\snc}{\mathsf{NC}}

\begin{document}

\title{The Second-Order Football-Pool Problem and the Optimal Rate of Generalized-Covering Codes\tnoteref{t1}}
\tnotetext[t1]{This work was supported in part by the German Israeli Project Cooperation (DIP) under Grant PE2398/1-1.}
\author[1]{Dor Elimelech\corref{cor1}}
\ead{doreli@post.bgu.ac.il}

\author[1]{Moshe Schwartz}
\ead{schwartz@ee.bgu.ac.il}

\cortext[cor1]{Corresponding author}
\address[1]{School of Electrical and Computer Engineering,
  Ben-Gurion University of the Negev,\\
  Beer Sheva 8410501, Israel}

\begin{abstract}

  The goal of the classic football-pool problem is to determine how many lottery tickets are to be bought in order to guarantee at least $n-r$ correct guesses out of a sequence of $n$ games played. We study a generalized (second-order) version of this problem, in which any of these $n$ games consists of two sub-games. The second-order version of the football-pool problem is formulated using the notion of generalized-covering radius, recently proposed as a fundamental property of linear codes. We consider an extension of this property to general (not necessarily linear) codes, and provide an asymptotic solution to our problem by finding the optimal rate function of second-order covering codes given a fixed normalized covering radius. We also prove that the fraction of second-order covering codes among codes of sufficiently large rate tends to $1$ as the code length tends to $\infty$.
\end{abstract}

\maketitle

\section{Introduction}
The covering problem is a fundamental problem in metric spaces: given a non-negative number $r$, find a set of points in the space that is of minimal size, such that the balls of radius $r$ centered those points cover the entire space. Such sets, often referred to as covering codes, have been thoroughly studied due to their fascinating relations with various topics in pure and applied mathematics, such as finite fields, discrete geometry, linear algebra, communication and algorithms. We refer to the excellent book~\cite{Cohen} for further reading on covering codes and their applications. 

The covering problem in Hamming spaces is frequently referred to as the football-pool problem (e.g., see~\cite{kamps1967football,linderoth2009improving,van1989new,wille1987football,ostergaard1994new,hamalainen1995football}), a name derived from a lottery-type gamble in which the outcomes of a sequence of football games are guessed. The football-pool problem deals with the following question: what is the minimal number of lottery tickets to be bought in order to guarantee that at least one of the tickets wins, where $n$ football games are played, and a ticket with at least $n-r$ games guessed correctly wins. The answer to that question is the minimal size of a covering code of length $n$ with covering radius at most $r$ in the Hamming space over an alphabet of $q$ elements (where $q$ is the number possible outcomes in a single football game).

We consider a generalization of the football-pool problem. Assume that $n$ football games are played, but now, each game consists of two sub-games, a match and a rematch, each with $q$ possible outcomes. A gambler buys lottery tickets with $n$ guesses, (each from the $q$ possible outcomes of a single match). The gambler is considered to guess correctly the $\ell$ games $i_1,\dots,i_{\ell}$ if they have two tickets, where the first ticket  guesses correctly the first matches in the games $i_1,\dots,i_{\ell}$ and the second ticket guesses correctly the the rematches of the same games $i_1,\dots,i_{\ell}$ (the same ticket may be used twice). The gambler wins if they can guess correctly at least $n-r$ games. The goal in our generalized (second-order) football-pool problem is to determine what is the minimal number of tickets to be bought in order to guarantee winning.

Similarly to the original football-pool problem, this minimal number of tickets would be the minimal size of a \emph{second-order covering code} of length $n$ with \emph{second covering radius} at most $r$ in the Hamming space over an alphabet of size $q$. Motivated by this generalized version of the football-pool problem, we study generalized covering codes in Hamming spaces.

The generalized covering radius was recently introduced as a fundamental property of linear codes, shown to characterize a trade-off between access-complexity, storage and latency in linear data-querying protocols. In~\cite{elimelech2021generalized}, the case of linear codes was studied: some fundamental properties of the generalized covering radii were examined, and asymptotic bounds on the optimal rates of linear covering codes were derived. An interesting relation between the generalized covering radius and generalized Hamming weights of linear codes (see~\cite{1991-Wei}) was also observed.  In another paper~\cite{elimelech2022generalized}, the generalized covering radii of Reed-Muller codes were examined. 

In this work, we focus on the generalized covering radii of general codes, i.e., codes which are not necessarily linear. The main result in the paper is the derivation of the exact value of the minimal asymptotic rate of second-order covering codes with a fixed normalized second covering radius over an arbitrary finite alphabet. For a normalized radius  $\rho\in [0,1]$,  denoting the second-order optimal rate  function over an alphabet of size $q$ by $\kappa_2(\rho,q)$, we prove in Theorem~\ref{th:BinRatecovering} that 
\[\kappa_2(\rho,q)=\begin{cases}
1-H_{q^2}(\rho) & \rho\in [0, 1-\frac{1}{q^2} ),\\
0 & \rho\in [1-\frac{1}{q^2},1],
\end{cases}\]
where $H_{q^2}(\cdot)$ denotes the $q^2$-ary entropy function. This result is an improvement upon the best known upper bound on the minimal asymptotic rate of linear binary second-order covering codes, given in \cite[Theorem 22]{elimelech2021generalized}. Thus, while a gap still remains for linear codes, our main result for general codes completely finds $\kappa_2(\rho,q)$, while also extending to general finite alphabets.

Another important result in this paper is given in   Theorem~\ref{th:FractionCovering}, where we prove that second-order covering codes are very common among codes of sufficiently large rate. For $\rho\in [0,1-\frac{1}{q^2})$ let  $\alpha_q(n,\rho,M)$ denote the fraction of codes of length $n$ over an alphabet of size $q$ with normalized second covering radius at most $\rho$ in the set of $\parenv*{n,M}_q$ codes. In Theorem~\ref{th:FractionCovering} we prove that for any $\varepsilon>0$
\[ \lim_{n\to\infty}\alpha_q\parenv*{n,\rho,q^{n(1-H_{q^2}(\rho)+\varepsilon)}}=1.\]

\section{Preliminaries}
\label{sec:Pre}
We consider codes over finite Abelian groups. We use  $G_q$ to denote an Abelian group of size $q\in \N$ and $+$ for the group operation. Naturally, $G_q^n$ denotes the set of vectors of length $n$ with entries from $G_q$, and $G_q^{t\times n}$ denotes the set of $t\times n$ matrices with entries from $G_q$. We also consider $G_q^n$ and $G_q^{t\times n}$ as Abelian groups with the entry-wise group operation. We use lower-case letters, $v$, to denote scalars and group elements. Overlined lower-case letters, $\ov$, shall be used to denote vectors, and bold lower-case letters, $\bv$, to denote matrices. 

For a vector $\ov=(v_1,\dots,v_n)\in G_q^n$, the support of $\ov$ is defined as
\[\supp(\ov)\eqdef \set{1\leq i\leq n ; v_i\neq 0},\]
and its Hamming weight is defined as
\[\wt(\ov)\eqdef \abs{\supp(\ov)}.\]
The Hamming distance between two vectors $\ov,\ov'\in G_q^n$ is then defined as
\[d(\ov,\ov')\eqdef\wt(\ov'-\ov).\]

 A set $C\subseteq G_q^n$ is called an $(n,M)_q$ code if it has cardinality $M$. The elements in a code $C$ shall also be called codewords. For an $(n,M)_q$ code, $\log_q(M)$ is called the dimension of the code. In the case where $G_q$ is as also a field, we say that $C$ is a linear code if it is a linear subspace of $G_q^n$ over $G_q$. In that case, $C$ is said to be an $[n,k]_q$ linear code,  where $k=\log_q(M)$  is its dimension (which is also the dimension of $C$ as a vector space). 

For an $(n,M)_q$ code $C$, the covering radius of $C$, denoted $R(C)$, is the distance of the farthest point in $G_q^n$ to the code, with respect to the Hamming distance. That is,
\[R(C)\eqdef \max_{\ov\in G_q^n}\min_{\oc\in C}d(\oc,\ov).\]
Equivalently, the covering radius of the code is the minimum radius at which balls centered at the codewords of $C$ cover the entire space $G_q^n$. Here, a ball of radius $r$ (not necessarily an integer) centered at $\ov\in G_q^n$ is defined as the set of vectors in $G_q^n$ that are at distance no more than $r$ from $\ov$, i.e.,
\[B_{r}(\ov)\eqdef \set*{\ou\in G_q^{n}  ; d(\ov,\ou)\leq r}.\]
The normalized covering radius of $C$ is denoted by $\rho(C)$, and is defined to be
\[ \rho(C)\eqdef\frac{R(C)}{n}.\]


The generalized covering radius was introduced in~\cite{elimelech2021generalized} as a fundamental property of linear codes. While~\cite{elimelech2021generalized} only studied linear codes, we extend our view to general codes, i.e., codes which are not necessarily linear. We begin by recalling the definition of the $t$-metric, also known as the block metric, on the space of matrices $G_q^{t\times n}$.

\begin{definition} 
Let $\bv\in G_q^{t\times n}$ be a matrix with rows denoted by $\ov_1,\dots,\ov_t$. The $t$-weight of $\bv$ is defined by
\[\wt^{(t)}(\bv)\eqdef\abs*{\bigcup_{i\in [t]}\supp{\ov_i}}.\]
The $t$-distance between two matrices $\bv$ and $\bu$ in $G_q^{t\times n}$   is defined to be
\[d^{(t)}(\bv,\bu)\eqdef \wt^{(t)}(\bu-\bv).\]
The $t$-Ball is defined in the usual manner, with respect to the $t$-metric:
\[B_{r}^{(t)}(\bv)\eqdef \set*{\bu\in G_q^{t\times n}  ; d^{(t)}(\bv,\bu)\leq r}.\]
\end{definition}

We remark that for $t=1$, we get the well known Hamming metric. Thus, notationally, when $t=1$ we may omit the superscript $^{(1)}$. Next we define the $t$-th power of a code.

\begin{definition}\label{def:codeVariations}
 Let $C$ be an $(n,M)_q$ code and $t\in\N$. We define $C^t\subseteq G_q^{t\times n}$  to be the set of $t\times n$ matrices over $G_q$ such that their rows are codewords in $C$. That is,
\[C^t \eqdef \set*{ \begin{bmatrix} \oc_1 \\ \vdots \\ \oc_t\end{bmatrix} \in 
G_q^{t\times n}; \forall i\in[t], \oc_i\in C}.\]
\end{definition}

We are now ready to define the $t$-th-covering radius of a code.
\begin{definition}
Let $C$ be an $(n,M)_q$ code and $t\in\N$. The $t$-th-covering radius of $C$ is defined to be the (regular) covering radius of $C^t$ inside $G_q^{t\times n}$ with respect to the $t$-metric. That is,
\[R_t(C)\eqdef\max_{\bu\in G_q^{t\times n}}\min_{\bc\in C^t}d^{(t)}(\bc,\bu).\]
\end{definition}

Once again, we note that for $t=1$, the $t$-th-covering radius of a code is the regular well known covering radius of the code (with respect to the Hamming metric).

\begin{remark} In~\cite{elimelech2021generalized}, it is proved that in the case where $G_q$ is a finite field and $C$ is a linear code, the $t$-th-covering radius has several equivalent definitions, showing an algebraic aspect of this property. However, in the general case, where such an algebraic structure is missing, it is unclear if an extension of these equivalent definitions exists.
\end{remark}

\begin{remark}\label{rem:Group}
The definition of the $t$-th-covering radius depends on the $t$-metric, which is defined using the group operation. However, it is easy to check that the $t$-metric is invariant to a change of the group operation. Thus, the $t$-th-covering radius may be considered as a property of codes over arbitrary finite alphabets (by considering a finite alphabet of size $q$ as a cyclic group of order $q$). Nevertheless, for convenience and simplification of notation, we think of all codes as codes over finite Abelian groups.
\end{remark}


The fundamental problem in any coverings-type setting is to find the minimal size of a set with a covering radius which is at most $r$. Thus, we are interested in the minimal size (or equivalently, dimension or rate) of a code $C\subseteq G_q^n$ such that $R_t(C)\leq r$. 

\begin{definition}\label{def:RateFunc}
Let $n,t,q \in \N$, and $0\leq r \leq n$. The optimal dimension function, denoted by $k_t(n,r,q)$, is the minimal dimension of a code of length $n$ over a group of size $q$ with $t$-th-covering radius at most $r$. Namely, 
\[ k_t(n,r,q)\eqdef \min\set*{\log_q \abs*{C}  ; C\subseteq G_q^n, R_t(C)\leq r}. \] 
For $\rho \in [0,1]$, the asymptotic optimal rate is then defined as
\[ \kappa_t(\rho,q) \eqdef \liminf_{n\to \infty} \frac{ k_t(n,\rho n,q)}{n}.\] 
\end{definition}

We remark that the group $G_q$ is omitted from the notation, as by Remark~\ref{rem:Group}, $k_t$ and $\kappa_t$ only depend on the size $q$. 


A restriction to linear codes of the above functions was studied in~\cite{elimelech2021generalized}. Similarly to the general case, if $G_q=\F_q$ is the finite field of size $q$, then $k_t^{\mathrm{Lin}}$ and $\kappa_t^{\mathrm{Lin}}$ are defined to be 
 \[ k_t^{\mathrm{Lin}}(n,r,q)\eqdef \min\set*{\log_q \abs*{C}  ; \substack{C\subseteq \F_q^n, R_t(C)\leq r\\ C \text{ is linear }}}, \] 
and 
 \[ \kappa_t^{\mathrm{Lin}}(\rho,q) \eqdef \liminf_{n\to \infty} \frac{ k_t^{\mathrm{Lin}}(n,\rho n,q)}{n}.\] 
 Obviously, for all $n,t,r,\rho$  and prime power $q$ we have 
 \[  k_t(n,r,q) \leq  k_t^{\mathrm{Lin}}(n,r,q) \quad \text{and} \quad  \kappa_t(\rho,q) \leq  \kappa_t^{{\mathrm{Lin}}}(\rho,q).\] 

It is well known~\cite{cohen1985good} that in the case of $t=1$,
\begin{equation}
    \label{eq:1orderSol}
    \kappa_1(\rho,q)=  \kappa_1^{{\mathrm{Lin}}}(\rho,q)=\begin{cases}
1-H_q(\rho) & \rho\in[0,1-\frac{1}{q}),\\
0 & \rho\in[1-\frac{1}{q},1],
\end{cases}
\end{equation}
where $H_q$ is the $q$-ary entropy function defined by
\[ H_q(x) \eqdef x\log_q (q-1) - x\log_q (x) - (1-x)\log_q (1-x),\]
and for continuity, $H_q(0)\eqdef 0$.

At this point, our knowledge of $\kappa_t(\rho,q)$ becomes severely limited, and we restrict ourselves to the first unresolved case, i.e., $t=2$. The lower bound from~\cite[Proposition 12]{elimelech2021generalized} gives us:
\begin{equation}
\label{eq:lowerball}
    \kappa_2(\rho,q) \geq \begin{cases}
    1-H_{q^2}(\rho) & \rho\in[0,1-\frac{1}{q^2}),\\
    0 & \rho\in[1-\frac{1}{q^2},1].
    \end{cases}
\end{equation}
This bound is based on a simple ball-covering argument. We also remark that while~\cite{elimelech2021generalized} only considered linear codes, the proof for the bound does not use the linearity of the code in any way, and thus the bound applies not only to $\kappa^{\mathrm{Lin}}_2(\rho,q)$, but also to $\kappa_2(\rho,q)$. In the other direction, \cite{elimelech2021generalized} only managed to handle the further restricted case of $q=2$, and thus~\cite[Proposition 14 and Theorem 22]{elimelech2021generalized} proved two upper bounds which give us:
\begin{align}
    \kappa_2(\rho,2)\leq\kappa^{\mathrm{Lin}}_2(\rho,2) & \leq 1-H_2\parenv*{\frac{\rho}{2}}, \label{eq:uppertrivial} \\
    \kappa_2(\rho,2)\leq\kappa^{\mathrm{Lin}}_2(\rho,2) & \leq \begin{cases}
    1-(4H_4(\rho)-f(\rho)) & \rho\in[0,\frac{3}{4}),\\
    0 & \rho\in[\frac{3}{4},1], \label{eq:upperbetter}
    \end{cases}
\end{align}
where, for all $\rho\in[0,\frac{3}{4})$ we define
\begin{align*}
    f(\rho)&\eqdef H_2(s(\rho))+2s(\rho)+2(1-s(\rho))H_2\parenv*{\frac{\rho-s(\rho)}{1-s(\rho)}},\\
    s(\rho)&\eqdef \frac{1}{10}\parenv*{1+8\rho-\sqrt{1+16\rho-16\rho^2}}.
\end{align*}
The bounds of~\cite{elimelech2021generalized} are depicted in Figure~\ref{fig:comp}, and a gap between the lower and upper bounds is evident. Our main theorem, proved in the following section, closes the gap completely, while extending the setting to a general alphabet of size $q$, giving us the exact value of $\kappa_2(\rho,q)$.

\begin{figure}[t]
\centering
\begin{overpic}[scale=0.9]
    {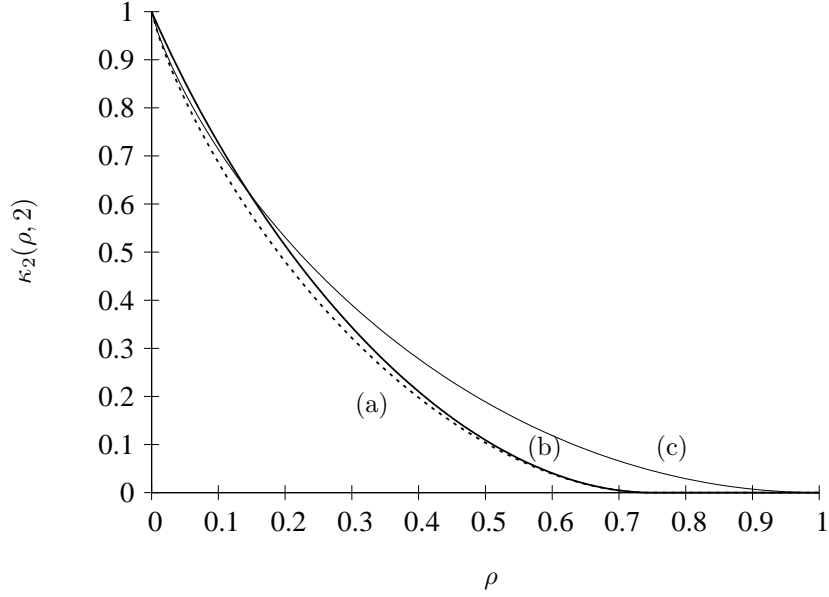}
    \put(0,35){\begin{turn}{90}$\kappa_2(\rho,2)$\end{turn}}
    \put(55,0){$\rho$}
    \put(40,20){(a)}
    \put(60,15){(b)}
    \put(75,15){(c)}
\end{overpic}
\caption{
A comparison of the bounds on $\kappa_2(\rho,2)$: (a) the ball-covering lower bound of~\eqref{eq:lowerball}, (b) the upper bound of \eqref{eq:upperbetter}, and (c) the upper bound of \eqref{eq:uppertrivial}.
}
\label{fig:comp}
\end{figure}

A key component in the proofs ahead is an estimate of the size of balls. Let $V_{r,n,q}^{(t)}$ denote the size of a $t$-ball of radius $r$ in $G_q^{t\times n}$ with respect to $d^{(t)}$,
\[ V_{r,n,q}^{(t)} \eqdef \abs*{B_r^{(t)}(\bv)},\]
which does not depend on the center, $\bv$, as the metric is translation invariant. By choosing $\bv=\Zero$, one can easily see that $V_{r,n,q}^{(t)}$ counts the number of $t\times n$ matrices with at most $r$ non-zero columns. Thus, after conveniently denoting $\rho=\frac{r}{n}$,
\[ V^{(t)}_{\rho n,n,q}=\sum_{i=0}^{\floor{\rho n}} \binom{n}{i}(q^t-1)^i=V^{(1)}_{\rho n,n,q^t}.
\]
By a standard use of Stirling's approximation (e.g., see~\cite[Chapter 3]{GurRudSud22}) it is well known that for $\rho\in \sparenv{0,1-\frac{1}{q^t}}$
\begin{equation} \label{eq:BallSizeIneq}
   q^{tn(H_{q^t}(\rho)-o(1))}\leq V^{(1)}_{\rho n,n,q^t}\leq  q^{tn H_{q^t}(\rho)},
\end{equation}
and therefore 
\begin{equation}
    \label{eq:BallEntropy}
    V^{(t)}_{\rho n,n,q}=\sum_{i=0}^{\floor{\rho n}} \binom{n}{i}(q^t-1)^i=V^{(1)}_{\rho n,n,q^t}
    = \begin{cases}
q^{tn(H_{q^t}(\rho)+o(1))} & \rho \in[0,1-\frac{1}{q^t}),\\
q^{tn(1-o(n))}& \rho\in[1-\frac{1}{q^t},1].
\end{cases}
\end{equation} Using the same approximation, we also mention that for $0\leq m\leq n$, $n>0$,
\begin{equation}
\label{eq:binomial}
\binom{n}{m}(q-1)^m = q^{n(H_q(m/n)+o(1))}.
\end{equation}
Finally, here in~\eqref{eq:BallEntropy}-\eqref{eq:binomial} and throughout the paper, we use $o(1)$ to denote a function of $n$ whose limit is $0$ as $n\to\infty$. Then, given a continuous real function $f(x)$, we shall often use the fact that $f(x+o(1))=f(x)+o(1)$. 

\section{The second-order optimal rate}\label{sec:SecondOreder}

The purpose of this section is to prove the following main theorem:

\begin{theorem}
\label{th:BinRatecovering}
\[\kappa_2(\rho,q)=\begin{cases}
1-H_{q^2}(\rho) & \rho\in[0,1-\frac{1}{q^2}), \\
0 & \rho\in[1-\frac{1}{q^2},1].
\end{cases}\]
\end{theorem}

Since the proof of Theorem~\ref{th:BinRatecovering} is long and involved, we first describe the overall strategy in brief. We start by noting that the lower bound of~\eqref{eq:lowerball} matches the claim of Theorem~\ref{th:BinRatecovering}. Additionally, the upper bound of~\eqref{eq:upperbetter} matches the claim of Theorem~\ref{th:BinRatecovering} in the range $[1-\frac{1}{q^2},1]$. Furthermore, the case of $\rho=0$ is trivial. Hence, it remains to prove an upper bound matching Theorem~\ref{th:BinRatecovering} in the interval $(0,1-\frac{1}{q^2})$.

In order to show that $\kappa_2(\rho,q)$ is upper bounded by some number $\gamma$, we are required to find a sequence of codes with lengths that tend to infinity, whose normalized second covering radius is no more then $\rho$, and whose rate (asymptotically) does not exceed $\gamma$. 

In order to find such codes, we take a probabilistic approach. We generate random codes using a carefully chosen distribution. Then, we prove that the event of obtaining a second-order covering code with a normalized radius not bigger than $\rho$, is non-zero for a large-enough length. We then make sure that some of these codes have a sufficiently low rate. This will imply that the desired codes exist and the upper bound holds.

From now on, we fix some $\rho\in (0,1-\frac{1}{q^2})$. Let $\set{\chi_{\ov}}_{\ov\in G_2^n}$ be a set of i.i.d $\mathrm{Ber}(p)$ random variables. We consider the random code $C$ which consists of all the vectors $\ov\in G_q^n$ such that $\chi_{\ov}=1$, i.e.,
\[ C\eqdef \set*{\ov\in G_q^n ; \chi_{\ov}=1}.\]

Let $\ou_1,\ou_2\in G_q^n$ be two vectors, and assume $\bv\in G_q^{2\times n}$. We say that the unordered pair $\set{\ou_1,\ou_2}$ covers $\bv$, denoted $\set{\ou_1,\ou_2}\supsetplus \bv$, if $\bv$ is contained in at least one of the two balls of radius $\rho n$ centered at $\begin{bsmatrix}
\ou_1\\ \ou_2
\end{bsmatrix}$
and
$\begin{bsmatrix}
\ou_2\\ \ou_1
\end{bsmatrix}$.
That is,
\[\set*{\ou_1,\ou_2}\supsetplus \bv \quad\text{ iff }\quad \bv\in B_{\rho n}^{(2)}(\begin{bsmatrix}
\ou_1\\ \ou_2
\end{bsmatrix})\cup B_{\rho n}^{(2)}(\begin{bsmatrix}
\ou_2\\ \ou_1
\end{bsmatrix} ) \text{ and } \ou_1\neq\ou_2.
\]
Equivalently,
\[\set*{\ou_1,\ou_2}\supsetplus \bv \quad\text{ iff }\quad \set*{\begin{bsmatrix}
\ou_1\\ \ou_2
\end{bsmatrix},\begin{bsmatrix}
\ou_2\\ \ou_1
\end{bsmatrix} }\cap B_{\rho n}^{(2)}(\bv)\neq \emptyset \text{ and } \ou_1\neq\ou_2.
\]
Then, for any matrix $\bv\in G_q^{2\times n}$ we define the random variable
\[ X_{\bv}\eqdef\sum_{\set*{\ou_1,\ou_2}\supsetplus \bv }\chi_{\ou_1}\cdot \chi_{\ou_2}. \]
We observe that if $X_{\bv}>0$ then $\bv$ is $2$-covered by at least one matrix from $C^2$ with distinct rows. 

Aiming for a lower bound on $\Prob[X_{\bv}=0]$, we use the Janson-type concentration inequality given as follows:

\begin{theorem}[{{\cite[Theorem 11]{schwartz2011new}}}]
\label{thm:MosheBound}
Let $\set{\chi_i}_{i\in\cQ}$ be a finite set of independent Boolean random variables, and let $\cA\subseteq 2^{\cQ}$ be a family of non-empty subsets. Let $X$ be the random variable defined by 
\[ X\eqdef \sum_{A\in \cA}I_A,\quad I_A\eqdef \prod_{i\in A}\chi_i, \]
and for each $A\in\cA$ let us define 
\[ X_A\eqdef I_A +\sum_{\substack{A \neq B\in \cA\\
  A\cap B \neq \emptyset}}I_B, \quad \text{and}\quad p_A\eqdef{\Prob[I_A=1]}.\] 
  Then,
\begin{equation}
    \label{eq:JansonB}
    \Prob[X=0]\leq \exp\parenv*{-\sum_{A\in \cA}p_A\E\sparenv*{\frac{1}{X_A} ; I_A=1}}.
\end{equation}
\end{theorem}

One can easily see that for any $\bv\in G_q^{2\times n}$, our probabilistic model exactly fits the setting of Theorem~\ref{thm:MosheBound} with 
\[ X=X_{\bv}, \quad \cQ=G_q^n, \quad \text{and} \quad \cA=\set*{\set*{\ou_1,\ou_2}  ; \ou_1\neq\ou_2 \text{ and }\set*{\ou_1,\ou_2}\supsetplus \bv }.\]

Given $A=\set{\ou_1,\ou_2}$, with $\ou_1,\ou_2\in G_q^n$, and given $\bv=\begin{bsmatrix}
\ov_1\\ \ov_2
\end{bsmatrix}\in G_q^{2\times n}$, we shall conveniently define
\begin{equation}
\label{eq:defwa}
w_A\eqdef \min_{i,j\in\set{1,2}}{d(\ou_i,\ov_j)}, \end{equation}
where the dependence on $\bv$ is implicit in the notation $w_A$.

 \begin{lemma}\label{lem:InverseBinomial}
 With the notation above, for any $\bv=\begin{bsmatrix}
 \ov_1\\ \ov_2
 \end{bsmatrix}\in G_q^{2\times n}$ and $A=\set{\ou_1,\ou_2}\in \cA$,  we have that 
 \[  \frac{1}{2}\cdot\frac{1-(1-p)^{n_A+1}}{p(n_A+1)}\leq\E\sparenv*{\frac{1}{X_A} ; I_A=1}\leq \frac{1-(1-p)^{n_A+1}}{p(n_A+1)} \]
 where $n_A$ is an integer satisfying 
 \[ q^{w_A}\cdot V^{(1)}_{\rho n-w_A,n-w_A,q}\leq n_A \leq  4\cdot q^{w_A}\cdot V^{(1)}_{\rho n-w_A,n-w_A,q}.\]
 \end{lemma}

\begin{proof}
Under the conditional measure given the event $\set{I_A=1}$, with probability $1$ the random variable $X_A$ is equal to the random variable $Z+1$, where $Z$ is given by 
\[ Z\eqdef \sum_{\ow\in G_q^{n}\setminus \set*{\ou_1,\ou_2}}\alpha(\ow) \chi_{\ow},\quad \text{where}\quad \alpha(\ow)\eqdef\ind_{\set*{\ow,\ou_1}\supsetplus \bv}+\ind_{\set*{\ow,\ou_2}\supsetplus \bv},\]
and where for an event $P$, $\ind_P$ denotes its corresponding indicator function.
 
For each $\ow\in G_q^n\setminus\set{\ou_1,\ou_2}$, we have $\alpha(\ow)\in \set{0,1,2}$, and therefore, for $\widetilde{Z}$ defined as 
\[\tZ\eqdef \sum_{\substack{\ow\in G_q^n\setminus \set*{\ou_1,\ou_2}\\ \alpha(\ow)\neq 0}}\chi_{\ow},\]
 we have
\[ \tZ \leq Z\leq 2\tZ. \]
In particular, under the conditional measure given the event $\set{I_A=1}$, with probability $1$ it holds that 
\[ \frac{1}{2}\cdot \frac{1}{\tZ+1}\leq \frac{1}{X_A}\leq  \frac{1}{\tZ+1}. \] 
By the monotonicity of the expectation, 
\[ \frac{1}{2}\cdot\E\sparenv*{\frac{1}{\tZ+1} ; I_A=1} \leq \E\sparenv*{ \frac{1}{X_A} ; I_A=1}\leq \E\sparenv*{\frac{1}{\tZ+1} ; I_A=1}. \]

We observe that $\tZ$ is a function of $\set{\chi_{\ow}}_{\ow\in G_q^n\setminus \set{\ou_1,\ou_2}}$ and $I_A$ is a function of $\chi_{\ou_1}$ and $\chi_{\ou_2}$. Hence, $\tZ$ is independent of $I_A$, which implies that 
\[\E\sparenv*{\frac{1}{\tZ+1} ; I_A=1} =\E\sparenv*{\frac{1}{\tZ+1}} . \]
Directly from its definition, we get that  $\tZ\sim \mathrm{Bin}(n_A,p)$ with 
\[n_A=\abs*{\set*{\ow\in G_q^{n}\setminus\set*{\ou_1,\ou_2}  ;\alpha(\ow)>0}}.\]
We now use the result given in \cite[Chapter 3.1, Eq.~(3.4)]{chao1972negative}, stating that if $Y\sim \mathrm{Bin}(n,p)$, then 
\[\E\sparenv*{\frac{1}{Y+1}}=\frac{1-(1-p)^{n+1}}{p(n+1)},\]
and conclude that 
\[\frac{1}{2}\cdot \frac{1-(1-p)^{n_A+1}}{p(n_A+1)}\leq\E\sparenv*{\frac{1}{X_A} ; I_A=1}\leq  \frac{1-(1-p)^{n_A+1}}{p(n_A+1)}.\]
  
In order to complete the proof it remains to bound $n_A$. We recall that $n_A$ is the number of vectors in $G_q^{n}\setminus\set{\ou_1,\ou_2}$ that together with  $\ou_1$ or $\ou_2$ can form a $2\times n$ matrix in $B_{\rho n}^{(2)}(\bv)$. We further sub-divide this set (perhaps with overlaps) in the following manner. For $i,j\in \set{1,2}$ we define $\phi_{i,j}(\ow)$ to be the $2\times n$ matrix whose $i$th row is $\ow$, and whose other row (the $(3-i)$th row) is $\ou_j$. We then define
\[ S(i,j) \eqdef \set*{\ow\in G_q^n\setminus\set*{\ou_1,\ou_2} ; d^{(2)}(\phi_{i,j}(\ow),\bv) \leq \rho n}.\]
Since we can flip simultaneously the order of rows in $\phi_{i,j}(\ow)$ and $\bv$ without affecting the distance between them, we can equivalently write,
\begin{equation}
\label{eq:Sij}
S(i,j) = \set*{\ow\in G_q^n\setminus\set*{\ou_1,\ou_2} ; d^{(2)}\parenv*{
\begin{bmatrix}
\ow \\ \ou_j
\end{bmatrix},
\begin{bmatrix}
\ov_i \\ \ov_{3-i}
\end{bmatrix}
} \leq \rho n}.
\end{equation}
By recalling the definition of $n_A$ one may easily observe that
\[ \max_{i,j\in\set{1,2}}{\abs*{S(i,j)}}\leq n_A\leq \sum_{i,j\in\set{1,2}}{\abs*{S(i,j)}}\leq 4\cdot \max_{i,j\in\set{1,2}}{\abs*{S(i,j)}}.  \] 

Let us now compute $\abs{S(i,j)}$ for any $i,j\in \set{1,2}$. For our convenience we denote $r\eqdef\floor{\rho n}$. First, if $d(\ou_j,\ov_{3-i})>r$, then by~\eqref{eq:Sij} we must have $\abs{S(i,j)}=0$. Otherwise, denote $m\eqdef d(\ou_j,\ov_{3-i})\leq r$. In that case, the choices for $\ow\in S(i,j)$ are exactly the following: In the $m$ positions where $\ou_j$ and $\ov_{3-i}$ differ, we can set $\ow$ arbitrarily. In the remaining $n-m$ positions of $\ow$ we copy the entries of $\ov_i$, but we may change the value of at most $r-m$ of those positions. Hence,
\[ \abs*{S(i,j)}=q^m\sum_{\ell=0}^{r-m}\binom{n-m}{\ell}(q-1)^\ell=q^m\cdot V^{(1)}_{r-m,n-m,q}. \]

We observe that the expression describing $\abs{S(i,j)}$ is monotone non-increasing in $m$. This might be proved by noting that when we change an entry in $\ou_j$ in one of the coordinates in $\supp(\ou_j-\ov_{3-i})$, and make it equal to its counterpart in $\ov_{3-i}$ (thereby decreasing $m$ by $1$), any vector that belonged to $S(i,j)$ before the change, still does after the change, and in particular the size of $S(i,j)$ does not decrease. This shows that 
\[ \max_{i,j\in\set{1,2}}{\abs*{S(i,j)}}=q^{w_A}\sum_{i=0}^{r-w_A}\binom{n-w_A}{i}(q-1)^i=q^{w_A}\cdot V^{(1)}_{r-w_A,n-w_A,q},\] 
where $w_A$ is defined in~\eqref{eq:defwa}.
\qed
 \end{proof}

By further analyzing the function
$\frac{1-(1-p)^{n_A+1}}{p(n_A+1)}$ from Lemma~\ref{lem:InverseBinomial}, we immediately arrive at the following corollary:

\begin{corollary}\label{cor:Ffunction}
For any $\begin{bsmatrix}
 \ov_1\\ \ov_2
 \end{bsmatrix}=\bv\in G_q^{2\times n}$ and $A=\set{\ou_1,\ou_2}\in \cA$,  with $w_A=m=\mu n\leq \rho n $ we have that 
 \[
    \E\sparenv*{\frac{1}{X_A} ; I_A=1}\geq \frac{1}{2}\cdot\frac{1-(1-p)^{q^{n\cdot (f(\mu)+\frac{3}{n})}}}{p\cdot q^{n\cdot (f(\mu)+\frac{3}{n})}}=\frac{1}{2}\cdot\frac{1-(1-p)^{q^{n\cdot (f(\mu)+o(1))}}}{p\cdot q^{n\cdot (f(\mu)+o(1))}},
\]
where 
\[
f(\mu)\eqdef \begin{cases}
1 & \mu\in[0,1-q(1-\rho)], \\
\mu+(1-\mu)  H_q\parenv*{\frac{\rho-\mu}{1-\mu}}& \mu\in(1-q(1-\rho),\rho].
\end{cases}
\]
\end{corollary}

\begin{proof}
We consider the function  $\frac{1-(1-p)^{n_A+1}}{p(n_A+1)}$ as a function of $n_A$. By standard analysis techniques, or by recalling its equivalent definition as an inverse moment of a $\mathrm{Bin}(n_A,p)$ random variable, we note that $\frac{1-(1-p)^{n_A+1}}{p(n_A+1)}$ is decreasing with $n_A$. We observe that for $\mu \in (1-q(1-\rho),\rho]$ we have 
$\frac{\rho-\mu}{1-\mu}\in [0,1-\frac{1}{q})$, and therefore by Lemma~\ref{lem:InverseBinomial} and \eqref{eq:BallSizeIneq} we have 
\begin{align*}
    n_A+1&\leq 1+4\cdot q^{m}\sum_{i=0}^{r-m}\binom{n-m}{i}(q-1)^i=1+4\cdot q^{m}\cdot V^{(1)}_{r-m,n-m,q} \\
    &\leq 5q^{m}\cdot V^{(1)}_{r-m,n-m,q}\leq q^{\mu n +3}q^{n(1-\mu)H_q\parenv*{\frac{\rho-\mu}{1-\mu}}}=q^{n(f(\mu)+\frac{3}{n})}.
\end{align*}

Combining the above inequality with the (decreasing) monotonicity and the lower-bound from Lemma~\ref{lem:InverseBinomial} 
\[ \E\sparenv*{\frac{1}{X_A} ; I_A=1}\geq \frac{1}{2}\cdot\frac{1-(1-p)^{q^{n\cdot (f(\mu)+\frac{3}{n})}}}{p\cdot q^{n\cdot (f(\mu)+\frac{3}{n})}}.\]
For $ \mu\in[0,1-q(1-\rho)] $ we have 
\[n_A+1\leq 5q^{m}\cdot V^{(1)}_{r-m,n-m,q}\leq 5q^n\leq q^{n f(\mu)+3}\leq q^{n(f(\mu)+\frac{3}{n})},\]
and the conclusion similarly follows.
\qed
\end{proof}
 
We now turn towards an asymptotic analysis of $\Prob[X_{\bv}=0]$. Our strategy is to show, using the Janson-type inequality given in Theorem~\ref{thm:MosheBound}, that for an appropriate choice of $p$, this probability decreases rapidly to $0$ for all the matrices in $G_q^{2\times n}$. Let $A=\set{\ou_1,\ou_2}$ be such that $w_A=m=\mu n\leq \rho n$. As we continue, we shall find the case of $\mu=\frac{q}{q+1}\rho$ of particular interest. In the following lemma we show the existence of a large subset $\cA'\subseteq \cA$ such that for all $A\in\cA'$ we have $w_A=n(\mu + o(1))$.

\begin{lemma} \label{lem:MinDistCount}
Let $\begin{bsmatrix}
\ov_1\\ \ov_2
\end{bsmatrix}=\bv\in G_q^{2\times n}$ be any matrix, $\rho\in(0,1-\frac{1}{q^2})$, and $\mu=\frac{q}{q+1}\rho$. Then there exists a subset $\cA'\subseteq\cA$ with
\[    \abs*{\cA'} \geq q^{n\parenv*{H_q(\mu)+\mu +(1-\mu)  H_q \parenv*{\frac{\rho-\mu}{1-\mu}}+o(1)}},\]
such that for all $A\in\cA'$ we have 
\[ \mu n -11 \leq w_A \leq \mu n.\]
\end{lemma}

\begin{proof}
Throughout the proof we shall occasionally use the fact that for any real $\alpha\in[0,1]$ and any integer $\ell$,
\[ \floor*{\alpha \ell}+\ceil*{(1-\alpha)\ell}=\ell.\]
Let $d\eqdef d(\ov_1,\ov_2)$ and denote $\delta\eqdef\frac{d}{n}$. By translation invariance and coordinate reordering, we may assume, without loss of generality, that $\ov_1=\ozero_n$ and $\ov_2=\ov'_2\ozero_{n-d}$, with $\ov'_2\in(G_q\setminus\set{0})^d$.

The proof strategy is to show the existence of sufficiently many elements $\set{\ou_1,\ou_2}=A\in\cA$ such that $w_A=n(\mu+o(1))$. These elements will form the set $\cA'$. We first choose $\ou_2$ to be the same as $\ov_2$, except that we change $\floor{\delta \floor{\mu n}}$ entries of the $\ov'_2$ part into other values, exactly $\floor{\frac{q-2}{q-1}\floor{\delta \floor{\mu n}}}$ of which are non-zero. We also change $\ceil{(1-\delta)\floor{\mu n}}$ entries of the $\ozero_{n-d}$ part into non-zero values. We emphasize that
\begin{align*}
    \floor*{\delta\floor*{\mu n}}&\leq \delta n,
    &
    \ceil*{(1-\delta)\floor*{\mu n}} &\leq (1-\delta) n,
\end{align*}
and so it is possible to choose that many coordinates. Thus, the number of ways for choosing $\ou_2$ in this fashion is
\begin{align}
&
\binom{\delta n}{\floor*{\delta \floor*{\mu n}}}
\binom{\floor*{\delta \floor*{\mu n}}}{\floor*{\frac{q-2}{q-1}\floor*{\delta \floor*{\mu n}}}} (q-2)^{\floor*{\frac{q-2}{q-1}\floor*{\delta \floor*{\mu n}}}}\nonumber\\
&\qquad \cdot\binom{(1-\delta)n}{\ceil*{(1-\delta)\floor*{\mu n}}} (q-1)^{\ceil*{(1-\delta)\floor*{\mu n}}} \nonumber\\
&\quad=\binom{\delta n}{\floor*{\delta \floor*{\mu n}}}
(q-1)^{\delta \mu n(1+o(1))}
\binom{(1-\delta)n}{\ceil*{(1-\delta)\floor*{\mu n}}} (q-1)^{(1-\delta)\mu n} \nonumber\\
&\quad=q^{\delta n \parenv*{H_q(\mu)+o(1)}}\cdot
q^{(1-\delta) n \parenv*{H_q(\mu)+o(1)}} \nonumber \\
&\quad=q^{n(H_q(\mu)+o(1))}, \label{eq:comb1}
\end{align}
where we used~\eqref{eq:binomial}, the continuity of the entropy function, and in particular when $q=2$ by convention we set $0^0=1$.

The set of coordinates in $\ou_2$ that started as non-zero and remained unchanged shall be denoted as $Z_{\snu}$, whereas those that were changed shall be denoted by $Z_{\snc}$. Similarly, the set of coordinates in $\ou_2$ that started as zero and remained unchanged shall be denoted as $Z_{\szu}$, whereas those that were changed shall be denoted by $Z_{\szc}$. The number of coordinates in each such set is then
\begin{align*}
\ell_{\snu} & = \delta n - \floor*{\delta\floor*{\mu n}} = \delta(1-\mu)n(1+o(1)), \\
\ell_{\snc} &= \floor*{\delta\floor*{\mu n}} = \delta\mu n(1+o(1)), \\
\ell_{\szc} &= \ceil*{(1-\delta)\floor*{\mu n}}=(1-\delta)\mu n (1+o(1)), \\
\ell_{\szu} &= (1-\delta)n - \ceil*{(1-\delta)\floor*{\mu n}}=(1-\delta)(1-\mu)n(1+o(1)).
\end{align*}
A schematic drawing is presented in Figure~\ref{fig:zones}.

\begin{figure}[t]
\centering
\begin{overpic}[scale=0.55]
    {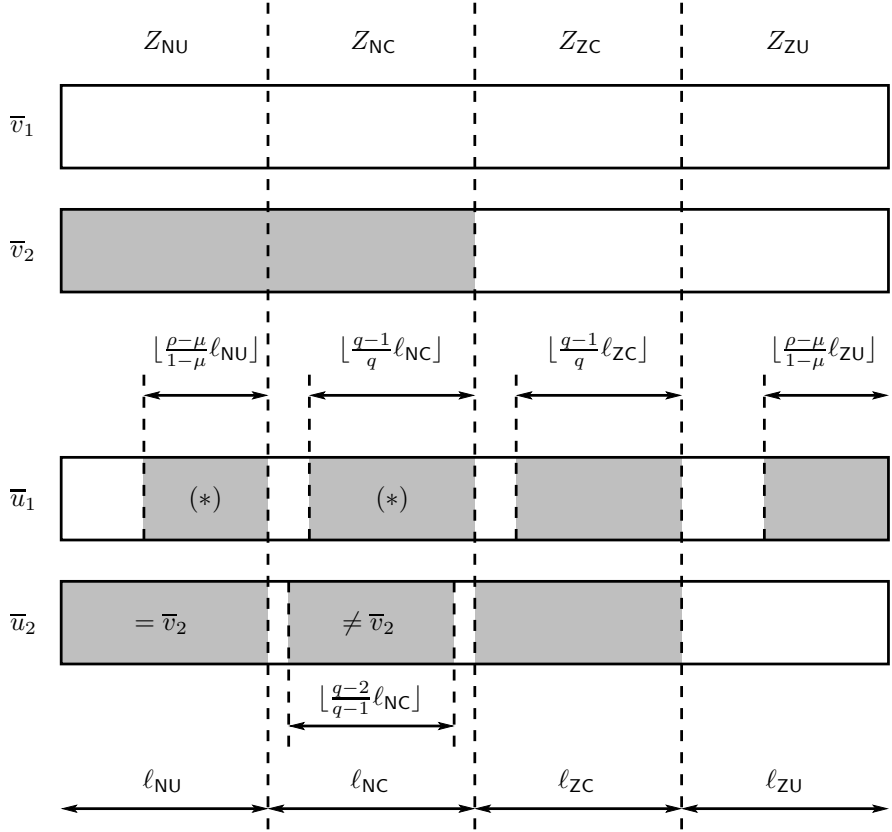}
    \put(-6,84){$\ov_1$}
    \put(-6,69){$\ov_2$}
    \put(-6,39){$\ou_1$}
    \put(-6,24){$\ou_2$}
    \put(10,94){$Z_{\snu}$}
    \put(35,94){$Z_{\snc}$}
    \put(60,94){$Z_{\szc}$}
    \put(85,94){$Z_{\szu}$}
    \put(10,5){$\ell_{\snu}$}
    \put(35,5){$\ell_{\snc}$}
    \put(60,5){$\ell_{\szc}$}
    \put(85,5){$\ell_{\szu}$}
    \put(31,15){$\floor{\frac{q-2}{q-1}\ell_{\snc}}$}
    \put(9,24){$=\ov_2$}
    \put(34,24){$\neq\ov_2$}
    \put(11,57){$\floor{\frac{\rho-\mu}{1-\mu}\ell_{\snu}}$}
    \put(33.5,57){$\floor{\frac{q-1}{q}\ell_{\snc}}$}
    \put(58.5,57){$\floor{\frac{q-1}{q}\ell_{\szc}}$}
    \put(85.5,57){$\floor{\frac{\rho-\mu}{1-\mu}\ell_{\szu}}$}
    \put(15.5,39){$(*)$}
    \put(38,39){$(*)$}
\end{overpic}
\caption{
The vectors $\ov_1,\ov_2,\ou_1,\ou_2$ in the proof of Lemma~\ref{lem:MinDistCount}. Shaded areas contain only non-zero entries, whereas non-shaded areas contain only zero entries. The notations $=\ov_2$ and $\neq\ov_2$ mean component-wise equality or inequality when comparing with the corresponding entries in $\ov_2$. Additionally, $(*)$ denotes that exactly a $\frac{q-2}{q-1}$-fraction of entries (rounded down) disagree with the corresponding entries in $\ov_2$.
}
\label{fig:zones}
\end{figure}

We now choose $\ou_1$ by describing where the non-zero elements are placed. In $Z_{\snu}$ we place $\floor{\frac{\rho-\mu}{1-\mu}\ell_{\snu}}$ non-zero elements, exactly a $\floor{\frac{q-2}{q-1}\floor{\frac{\rho-\mu}{1-\mu}\ell_{\snu}}}$ of which disagree with their corresponding elements in $\ov_2$. In $Z_{\snc}$ we place $\floor{\frac{q-1}{q}\ell_{\snc}}$ non-zero elements, exactly a $\floor{\frac{q-2}{q-1}\floor{\frac{q-1}{q}\ell_{\snc}}}$ of which disagree with their corresponding elements in $\ov_2$. In $Z_{\szc}$ we place $\floor{\frac{q-1}{q}\ell_{\szc}}$ non-zero elements, and in $Z_{\szu}$ we place $\floor{\frac{\rho-\mu}{1-\mu}\ell_{\szu}}$ non-zero elements. We again refer to Figure~\ref{fig:zones} for a schematic drawing. We observe that we can eliminate all of the floor and ceiling operations, and in return, multiply each expression by $(1+o(1))$.

It then follows that the total number of ways to choose $\ou_1$ in this fashion is\footnote{In the interest of having a readable expression, we removed all floor and ceiling operations, and notationally omitted multiplication by $(1+o(1))$, as it is absorbed in the $(1+o(1))$ from~\eqref{eq:binomial}.}
\begin{align}
&\binom{\delta(1-\mu)n}{\delta(\rho-\mu)n}\binom{\delta(\rho-\mu)n}{\frac{q-2}{q-1}\delta(\rho-\mu)n}(q-2)^{\frac{q-2}{q-1}\delta(\rho-\mu)n} && \text{(for $Z_{\snu}$)} \nonumber \\
&\quad\cdot
\binom{\delta \mu n}{\frac{q-1}{q}\delta\mu n}\binom{\frac{q-1}{q}\delta\mu n}{\frac{q-2}{q}\delta \mu n}(q-2)^{\frac{q-2}{q}\delta \mu n} && \text{(for $Z_\snc$)} \nonumber \\
&\quad\cdot
\binom{(1-\delta)\mu n}{\frac{q-1}{q}(1-\delta)\mu n}(q-1)^{\frac{q-1}{q}(1-\delta)\mu n} && \text{(for $Z_{\szc}$)} \nonumber \\
&\quad\cdot
\binom{(1-\delta)(1-\mu)n}{(1-\delta)(\rho-\mu)n}(q-1)^{(1-\delta)(\rho-\mu)n} && \text{(for $Z_{\szu}$)} \nonumber \\
& =
q^{\delta(1-\mu)n H_q\parenv*{\frac{\rho-\mu}{1-\mu}}(1+o(1))} \cdot
q^{\delta \mu n (1+o(1))} \nonumber \\
&\quad \cdot
q^{(1-\delta)\mu n (1+o(1))}
\cdot
q^{(1-\delta)(1-\mu)n H_q\parenv*{\frac{\rho-\mu}{1-\mu}}(1+o(1))} \nonumber \\
&=
q^{n\parenv*{\mu +(1-\mu)  H_q \parenv*{\frac{\rho-\mu}{1-\mu}}+o(1)}}.
\label{eq:comb2}
\end{align}
Again, we used~\eqref{eq:binomial} and the continuity of the entropy function.

Having constructed sets $A=\set{\ou_1,\ou_2}$, we turn to proving that they satisfy all the requirements. First, we examine $w_A$. We have the following inequalities: 
\begin{align*}
d(\ou_1,\ov_1) &= 
\floor*{\frac{\rho-\mu}{1-\mu}\parenv*{\delta n - \floor*{\delta\floor*{\mu n}}}} +\floor*{\frac{q-1}{q}\floor*{\delta\floor*{\mu n}}} \\
&\quad + \floor*{\frac{q-1}{q}\ceil*{(1-\delta)\floor*{\mu n}}} + \floor*{\frac{\rho-\mu}{1-\mu}\parenv*{(1-\delta)n - \ceil*{(1-\delta)\floor*{\mu n}}}} \\
&\geq \frac{\rho-\mu}{1-\mu}\parenv*{\delta n - \delta\mu n} + \frac{q-1}{q}\delta\mu n \\
&\quad + \frac{q-1}{q}(1-\delta)\mu n + \frac{\rho-\mu}{1-\mu}\parenv*{(1-\delta)n - (1-\delta)\mu n} - 8 \\
&= \mu n + \parenv*{\rho-\frac{q+1}{q}\mu}n - 8 = \mu n - 8,\\
d(\ou_1,\ov_2) & = \ceil*{\frac{1-\rho}{1-\mu}\parenv{\delta n - \floor*{\delta\floor*{\mu n}}}} + \floor*{\frac{q-2}{q-1}\floor*{\frac{\rho-\mu}{1-\mu}\parenv{\delta n - \floor*{\delta\floor*{\mu n}}}}} \\
&\quad + \ceil*{\frac{1}{q}\floor*{\delta\mu n}} + \floor*{\frac{q-2}{q-1}\floor*{\frac{q-1}{q}\floor*{\delta\floor*{\mu n}}}} + \floor*{\frac{q-1}{q}\ceil*{(1-\delta)\floor*{\mu n}}} \\
&\quad + 
\floor*{\frac{\rho-\mu}{1-\mu}\parenv*{(1-\delta)n - \ceil*{(1-\delta)\floor*{\mu n}}}} \\
&\geq \delta(1-\rho)n + \frac{q-2}{q-1}\delta(\rho-\mu)n + \frac{1}{q}\delta\mu n + \frac{q-2}{q}\delta\mu n \\
&\quad + \frac{q-1}{q}(1-\delta)\mu n + (1-\delta)(\rho-\mu)n - 11 \\
& = \mu n + \parenv*{1-\frac{q^2}{q^2-1}\rho}\delta n - 11 
\overset{(a)}{\geq} \mu n - 11,
\\
d(\ou_2,\ov_1) & = \delta n - \floor*{\delta\floor*{\mu n}}+ \floor*{\frac{q-2}{q-1}\floor*{\delta\floor*{\mu n}}}+\ceil*{(1-\delta)\floor*{\mu n}}\\
&\geq (1-\mu)\delta n+\frac{q-2}{q-1}\delta\mu n+(1-\delta)\mu n - 4 \\
& = \mu n + \parenv*{1-\frac{q^2}{q^2-1}\rho}\delta n - 4 
\overset{(a)}{\geq} \mu n - 4,
\\
d(\ou_2,\ov_2) & = \floor*{\delta\floor*{\mu n}} + \ceil*{(1-\delta)\floor*{\mu n}} = \floor*{\mu n},
\end{align*}
where (a) follows from $\rho\in(0,1-\frac{1}{q^2})$, and throughout we use the fact that $\mu=\frac{q}{q+1}\rho$. Combining all of the above we get,
\[ \mu n - 11 \leq w_A \leq \mu n.\]
Additionally,
\begin{align*}
d^{(2)}\parenv*{
\begin{bmatrix}
\ov_1 \\ \ov_2
\end{bmatrix},
\begin{bmatrix}
\ou_1 \\ \ou_2
\end{bmatrix}
} &= \floor*{\frac{\rho-\mu}{1-\mu}\parenv*{\delta n - \floor*{\delta\floor*{\mu n}}}}+\floor*{\delta\floor*{\mu n}}+\ceil*{(1-\delta)\floor*{\mu n}} \\
&\quad +\floor*{\frac{\rho-\mu}{1-\mu}\parenv*{(1-\delta)n - \ceil*{(1-\delta)\floor*{\mu n}}}} \\
&\leq \frac{\rho-\mu}{1-\mu}\parenv*{n-\floor*{\mu n}}+\floor*{\mu n} = \frac{\rho-\mu}{1-\mu}n + \frac{1-\rho}{1-\mu}\floor*{\mu n} \\
&\leq \frac{\rho-\mu}{1-\mu}n + \frac{1-\rho}{1-\mu}\mu n = \rho n.
\end{align*}
Hence, $\set{\ou_1,\ou_2}\supsetplus \bv$, and so $\set{\ou_1,\ou_2}\in \cA$.

We observe that if we account for the possibility of sometimes getting, $\ou_1=\ou_2$, and the possibility of getting twice the unordered pair $\set{\ou_1,\ou_2}$, by combining~\eqref{eq:comb1} and~\eqref{eq:comb2} we get
\begin{align*}
\abs*{\cA'}
& \geq \frac{1}{2} \cdot q^{n(H_q(\mu)+o(1))} \cdot \parenv*{q^{n\parenv*{\mu +(1-\mu)  H_q \parenv*{\frac{\rho-\mu}{1-\mu}}+o(1)}}-1} \\
&=q^{n\parenv*{H_q(\mu)+\mu +(1-\mu)  H_q \parenv*{\frac{\rho-\mu}{1-\mu}}+o(1)}},
\end{align*}
as claimed.
\qed
\end{proof}


Another technical result we shall need is the following entropy identity.

\begin{lemma}\label{lem:EntIdentity}
For any $\rho\in (0,1-\frac{1}{q^2})$ and $\mu=\frac{q}{q+1}\rho$,
\[ H_q(\mu) +\mu +(1-\mu)  H_q \parenv*{\frac{\rho-\mu}{1-\mu}} =2H_{q^2}(\rho).\] 
\end{lemma}
\begin{proof}
The proof is straightforward from the definition of the entropy function and the properties of the $\log$ function. For $\mu=\frac{q}{q+1}\rho$,
\begin{align*}
H_q&(\mu)+ \mu +(1-\mu)  H_q \parenv*{\frac{\rho-\mu}{1-\mu}}\\&=H_q\parenv*{\frac{q\rho}{q+1} }+\frac{q\rho}{q+1} +\parenv*{1-\frac{q\rho}{q+1}}  H_q \parenv*{\frac{\rho -\frac{q\rho}{q+1}}{1-\frac{q\rho}{q+1}}}\\
&=-\frac{q\rho}{q+1}\log_q\parenv*{\frac{q\rho}{q+1}}-\frac{q(1-\rho)+1}{q+1}\log_q\parenv*{\frac{q(1-\rho)+1}{q+1}}\\
&\qquad +\frac{q\rho}{q+1}\log_q(q-1)+\frac{q\rho}{q+1}+\parenv*{\frac{q(1-\rho)+1}{q+1}}H_q\parenv*{\frac{\rho}{q(1-\rho)+1}}\\
&=-\frac{q\rho}{q+1}\parenv*{-\log_q(\rho)-1+\log_q(q+1)+\log_q(q-1)}+ \frac{q\rho}{q+1}\\
&\qquad  -\frac{q(1-\rho)+1}{q+1}\parenv*{\log_q(q(1-\rho)+1)-\log_q(q+1)}\\
&\qquad +\frac{q(1-\rho)+1}{q+1}\parenv*{ -\frac{\rho}{q(1-\rho)+1}\log_q\parenv*{\parenv*{\frac{\rho}{q(1-\rho)+1}} +\log_q(q-1)} }\\
&\qquad - \frac{q(1-\rho)+1}{q+1}\parenv*{\frac{(q+1)(1-\rho)}{q(1-\rho)+1}\log_q\parenv*{\frac{(q+1)(1-\rho)}{q(1-\rho)+1}}}
\\&= \frac{q\rho}{q+1}\parenv*{-\log_q(\rho)+\log_q(q+1)+\log_q(q-1)}\\
&\qquad  -\frac{q(1-\rho)+1}{q+1}\parenv*{\log_q(q(1-\rho)+1)-\log_q(q+1)}\\
&\qquad +\frac{\rho}{q+1}\parenv*{\log_q(q-1)-\log_q(\rho)+\log_q(q(1-\rho)+1)}\\
&\qquad -(1-\rho)\parenv*{\log_q(q+1)+\log_q(1-\rho)-\log_q(q(1-\rho)+1)}\\
&=-\rho\log_q(\rho)-(1-\rho)\log_q(1-\rho)+\rho\log_q(q+1)+\rho\log_q(q-1)\\
&=2\parenv*{-\rho\log_{q^2}(\rho)-(1-\rho)\log_{q^2}(1-\rho)+\rho\log_{q^2}(q^2-1) }=2H_{q^2}(\rho).
\end{align*}
\qed
\end{proof}

\begin{lemma}\label{lem:EntIneq}
For any integer $q\geq 2$ and $\rho\in \parenv{0,1-\frac{1}{q^2}}$,
\[H_q\parenv*{\frac{q}{q+1}\rho }-H_{q^2}(\rho)> 0,\]
and in particular, the interval $\parenv{0,H_q\parenv{\frac{q}{q+1}\rho }-H_{q^2}(\rho)}$ is non-empty.
\end{lemma}

\begin{proof}
Let us investigate the function
\[\varphi(\rho)\eqdef H_q\parenv*{\frac{q}{q+1}\rho }-H_{q^2}(\rho)\]
as a function of $\rho\in\sparenv{0,1-\frac{1}{q^2}}$. We start by observing that
\begin{align*}
\varphi(0)&=H_{q}(0)-H_{q^2}(0)=0-0=0,\\
\varphi\parenv*{1-\frac{1}{q^2}}&=H_q\parenv*{1-\frac{1}{q}}-H_{q^2}\parenv*{1-\frac{1}{q^2}}=1-1=0.
\end{align*}
Thus, since $\varphi$ is smooth in $\parenv{0,1-\frac{1}{q^2}}$, in order to prove that $\varphi$ is positive, it is sufficient to show that $\varphi$ is increasing in a neighborhood of $0$ and that its first derivative has exactly one root in $\parenv{0,1-\frac{1}{q^2}}$.

A straightforward calculation of the first and second derivatives shows that 
\begin{align*}
    \varphi'(\rho) &=\frac{q}{(q+1)\ln(q)}\parenv*{\ln(q-1)-\ln\parenv*{\frac{q}{q+1}\rho}+\ln\parenv*{1-\frac{q}{q+1}\rho}}\\
    &\quad + \frac{1}{\ln\parenv*{q^2}}\parenv*{\ln(\rho)-\ln(1-\rho)-\ln\parenv*{q^2-1}},\\
    \varphi''(\rho) &=
    \frac{\rho\parenv*{q\ln(q)-q\ln(q^2)}-(q+1)\ln(q)+q\ln(q^2)}{\rho(1-\rho)(q(\rho-1)-1)\ln(q)\ln(q^2)}.
\end{align*}
We note that 
\begin{align*}
    \lim_{\rho\to 0_+}\varphi'(\rho)&=\lim_{\rho\to 0_+}\ln(\rho)\parenv*{-\frac{q}{(q+1)\ln(q)}+\frac{1}{\ln(q^2)}}\\
    &=\lim_{\rho\to 0_+}\frac{\ln(\rho)}{\ln(q)}\parenv*{\frac{1}{2}-\frac{1}{1+\frac{1}{q}}}=\infty,
\end{align*}
since $q\geq 2$. This proves that $\varphi$ is increasing in a neighborhood of $0$. Since $\varphi(0)=0$ it also implies that $\varphi$ is positive in a neighborhood of $0$.

It now remains to prove that $\varphi$ has exactly one root in $\parenv{0,1-\frac{1}{q^2}}$. So far we have shown that $\varphi$ is smooth, positive in a neighborhood of $0$, and satisfies $\varphi(0)=\varphi(1-\frac{1}{q^2})=0$, which together imply that it has at least one local extremum in $\parenv{0,1-\frac{1}{q^2}}$. This proves that $\varphi'$ has at least one root in $\parenv{0,1-\frac{1}{q^2}}$. We note that $\varphi'$ is also smooth, and therefore the number of roots of $\varphi'$ is upper-bounded by $1$ plus the number of its local extrema.

We also observe that the equation $\varphi''(\rho)=0$ has exactly one solution, 
\[\rho_0=\frac{\ln(q)+q\ln(q)-\ln(q^2)}{q(\ln(q)-\ln(q^2))}=1-\frac{1}{q}\in\parenv*{0,1-\frac{1}{q^2}}.\]
In particular, the number of extrema of $\varphi'$ is at most one, and therefore $\varphi'$ has at most two roots. 

By now, we know that $\varphi'$ has least one root and at most two roots in $\parenv{0, 1-\frac{1}{q^2}}$. We assume to the contrary that $\varphi'$ has two roots in $\parenv{0,1-\frac{1}{q^2}}$. In that case, $1-\frac{1}{q}$ must be a local extremum. Furthermore, since $\lim_{\rho\to 0_+}\varphi'(\rho)=\infty$, $1-\frac{1}{q}$ has to be a local minimum,  $\varphi'$ must be decreasing in $\parenv{0,1-\frac{1}{q}}$ and increasing in $\parenv{1-\frac{1}{q},1-\frac{1}{q^2}}$. Let $\rho_1,\rho_2$ be the roots of $\varphi'$, $\rho_1<1-\frac{1}{q}<\rho_2$. Since $\varphi'$ is increasing in $\parenv{1-\frac{1}{q},1-\frac{1}{q^2}}$, we have that
\[\lim_{\rho\to \parenv*{1-\frac{1}{q^2}}_-}\varphi'(\rho)>0.\]
On the other hand, $\varphi'$ naturally (and continuously) extends to the interval $(0,1)$ and
\begin{align*}
    \lim_{\rho\to \parenv*{1-\frac{1}{q^2}}_-}\varphi'(\rho)&=\varphi'\parenv*{1-\frac{1}{q^2}}\\
&=\frac{q}{(q+1)\ln(q)}\parenv*{\ln(q-1)-\ln\parenv*{\frac{q-1}{q}}+\ln\parenv*{\frac{1}{q}}}\\
    &\quad + \frac{1}{\ln\parenv*{q^2}}\parenv*{\ln\parenv*{\frac{q^2-1}{q^2}}-\ln\parenv*{\frac{1}{q^2}}-\ln\parenv*{q^2-1}}=0.
\end{align*}
This brings us to a contradiction and therefore completes the proof.
\qed
\end{proof}

We now have all the technical lemmas needed to bound $\Prob[X_{\bv}=0]$.

\begin{proposition}
\label{prop:ProbBoundWord}
Let $\rho\in \parenv{0,1-\frac{1}{q^2}}$ and $\varepsilon \in (0,H_q\parenv{\frac{q}{q+1}\rho}-H_{q^2}(\rho))$ be fixed. Assume that $p=q^{-n(H_{q^2}(\rho)-\varepsilon)}$, $\begin{bsmatrix}
\ov_1\\ \ov_2
\end{bsmatrix}=\bv\in G_q^{2\times n}$. Then, 
\[\Prob[X_{\bv}=0]\leq \exp\parenv*{-q^{n(2\varepsilon+o(1))}},\]
where the $o(1)$ term does not depend on $\bv$.

\end{proposition}

\begin{proof}
The components of the proof of the statement are the Janson-type inequality from Theorem~\ref{thm:MosheBound}, the tight estimations on $\E\sparenv{X_{A}^{-1} ; I_A=1}$ given in Lemma~\ref{lem:InverseBinomial}, and the combinatorial argument given in Lemma~\ref{lem:MinDistCount}. By Lemma~\ref{lem:MinDistCount}, for $\mu=\frac{q}{q+q}\rho$, there exists a subset $\cA'\subseteq \cA$ such that
\begin{equation}
    \label{eq:A'Size}
    \abs*{\cA'} \geq q^{n\parenv*{H_q(\mu)+\mu +(1-\mu)  H_q \parenv*{\frac{\rho-\mu}{1-\mu}}+o(1)}},
\end{equation} and for any $A\in \cA'$,  \begin{equation}\label{eq:w_Asize}
    \mu n - 11 \leq w_A \leq \mu n,
\end{equation} where the $o(1)$ term only depends on $\rho$ and $q$. Using Corollary~\ref{cor:Ffunction} and the fact that all summands in~\eqref{eq:JansonB} are non-negative, we obtain 
\begin{align*}
    \Prob\sparenv*{X_{\bv}=0}&\leq \exp\parenv*{-\sum_{A\in \cA}p_A \E\sparenv*{\frac{1}{X_A} ; I_A=1}}\\
    &= \exp\parenv*{-\sum_{A\in \cA}p^2 \E\sparenv*{\frac{1}{X_A} ; I_A=1}}\\
    &\leq  \exp\parenv*{-\sum_{A\in \cA
'}\frac{1}{2}\cdot p^2 \E\sparenv*{\frac{1}{X_A} ; I_A=1}}\\
    &= \exp\parenv*{-\sum_{A\in \cA
'} \frac{1}{2}\cdot p^2 \frac{1-(1-p)^{q^{n(f(w_A/n)+\frac{3}{n})}}}{p\cdot q^{n(f(w_A/n)+\frac{3}{n})}}}.
\end{align*}
Combining with \eqref{eq:w_Asize},
\begin{equation}
    \label{eq:onesummand}
\Prob\sparenv*{X_{\bv}=0} \leq \exp\parenv*{-\frac{1}{2}\sum_{A\in \cA
'} p^2 \frac{1-(1-p)^{q^{n(f(\mu+o(1))+o(1))}}}{p\cdot q^{n(f(\mu+o(1))+o(1))}}}
\end{equation}
Since $f(\cdot)$ is continuous, we have $f(\mu+o(1)) = f(\mu)+o(1)$, and the addition of $o(1)$ is easily absorbed in the $o(1)$ that already appears due to Corollary~\ref{cor:Ffunction}, still only depending on $q$ and $\rho$.

We begin our simplification of~\eqref{eq:onesummand} by using the well known Bernoulli's inequality: for $y>0$ and $x\in[-1,\frac{1}{y})$,
\[(1+x)^y\leq \frac{1}{1-xy}.\]
We use Bernoulli's inequality with $y=q^{n(f(\mu)+o(1))}>0$ and $x=-p\in [-1,0]$, to get
\begin{align*}
    (1-p)^{q^{n(f(\mu)+o(1))}}\leq \frac{1}{1+pq^{n(f(\mu)+o(1))}}.
\end{align*}
Therefore 
\begin{align}
      p \frac{1-(1-p)^{q^{n(f(\mu)+o(1))}}}{ q^{n(f(\mu)+o(1))}} & \geq   p \frac{1-\frac{1}{1+pq^{n(f(\mu)+o(1))}}}{ q^{n(f(\mu)+o(1))}} \nonumber \\
    & =  p \frac{pq^{n(f(\mu)+o(1))}}{ q^{n(f(\mu)+o(1))}\cdot \parenv*{1+pq^{n(f(\mu)+o(1))}}} \nonumber \\
    &=  \frac{p^2}{  1+pq^{n(f(\mu)+o(1))}}. \label{eq:cont1}
\end{align}
Since $\mu=\frac{q}{q+1}\rho$, we observe that for any value of $\rho\in (0,1-\frac{1}{q^2})$, have $\mu\in(1-q(1-\rho),\rho]$.
Recalling the definition of $f(\mu)$ Corollary~\ref{cor:Ffunction}, we have that 
\begin{align*}
1+pq^{n(f(\mu)+o(1))}&=
1+q^{n\parenv*{\mu+(1-\mu)H_q\parenv*{\frac{\rho-\mu}{1-\mu}}-H_{q^2}(\rho)+\varepsilon+o(1)}} \\
    &=1+q^{n\parenv*{H_{q^2}(\rho)-H_q\parenv*{\frac{q}{q+1}\rho}+\varepsilon+o(1)}},
\end{align*}
where the last equality follows from Lemma~\ref{lem:EntIdentity}.
We recall that by Lemma~\ref{lem:EntIneq} for any $\rho\in \parenv{0,1-\frac{1}{q^2}}$ 
\[ H_{q^2}(\rho)-H_q\parenv*{\frac{q}{q+1}\rho}<0,\]
which together with the assumption that $\varepsilon\in\parenv{0,H_q\parenv{\frac{q}{q+1}\rho}-H_{q^2}(\rho)}$ implies that 
\[
    1+pq^{n(f(\mu)+o(1))}=1+q^{n\parenv*{H_{q^2}(\rho)-H_q\parenv*{\frac{q}{q+1}\rho}+\varepsilon+o(1)}}
    =1+o(1).
\]
Continuing~\eqref{eq:cont1} and using \eqref{eq:A'Size} as well as Lemma~\ref{lem:EntIdentity}, we now conclude that for $\mu=\frac{q}{q+1}\rho$
\begin{align*}
      \sum_{A\in\cA'} p \frac{1-(1-p)^{q^{n(f(\mu)+o(1))}}}{ q^{n(f(\mu)+o(1))}} &\geq   \abs{\cA'}p^2(1+o(1))\\
    &= \abs{\cA'}q^{n(-2H_{q^2}(\rho)+2\varepsilon+o(1))}\\
    &\geq q^{n\parenv*{H_q(\mu)+\mu+(1-\mu)H_q\parenv*{\frac{\rho-\mu}{1-\mu}}-2H_{q^2}(\rho)+2\varepsilon+o(1)}}\\
    &=q^{n(2\varepsilon+o(1))}.
\end{align*}

Combining all the parts of the proof together, for all sufficiently large $n$, \eqref{eq:onesummand} gives
\begin{align*}
    \Prob\sparenv*{X_{\bv}=0}& \leq \exp\parenv*{-\frac{1}{2}\sum_{A\in A'}  p^2 \frac{1-(1-p)^{q^{n(f(\mu)+o(1))}}}{p\cdot q^{n(f(\mu)+o(1))}}} \leq \exp\parenv*{-q^{n\parenv*{2\varepsilon+o(1)}}},
 \end{align*}
where the $o(1)$ term is induced from $o(1)$ terms all along the way, depending only on $\rho,\varepsilon$ and $q$. 
\qed
\end{proof}

The last components for the proof of Theorem~\ref{th:BinRatecovering} are the following concentration inequality for binomial random variables, and a proposition that shows that, with high probability, the cardinality of the random code $C$ is close to its expected value of $q^{n(1-H_{q^2}(\rho)+\varepsilon)}$.

\begin{lemma}[{{\cite[Theorem 1]{janson2016large}}}]
\label{lem:BinBound}
Let $X\sim\mathrm{Bin}(n,p)$ be a binomial random variable. Then for every real $a>0$
\[\Prob[X\geq \E[X]+a]\leq \exp\parenv*{-\frac{a^2}{2np}\parenv*{1-\frac{a}{3np}}}.\]
\end{lemma}
 
From Lemma~\ref{lem:BinBound} it follows that if $X\sim\mathrm{Bin}(n,p)$, and $\gamma>0$, then 
\begin{equation}
     \label{eq:BinBound}\Prob[X\geq \E[X](1+\gamma)] = \Prob[X\geq np(1+\gamma)] \leq \exp\parenv*{-\frac{1}{2}\gamma^2 np\parenv*{1-\frac{\gamma}{3}}}.
\end{equation}

\begin{proposition}
\label{prop:RandomCodingGood}
Under the assumptions of Proposition~\ref{prop:ProbBoundWord}, for
all sufficiently large $n$,
\[ \Prob\sparenv*{\set*{R_2(C)\leq \rho n}\cap \set*{\abs{C}< q^{n(1-H_{q^2}(\rho) +\varepsilon+1/n)}}}\geq 1-\exp\parenv*{-q^{n(2\varepsilon+o(1))}}.\]
\end{proposition}

\begin{proof}
From Proposition~\ref{prop:ProbBoundWord}, for any matrix $\bv\in G_q^{2\times n}$, 
\[\Prob[X_{\bv}=0]\leq \exp\parenv*{-q^{n(2\varepsilon+o(1))}}.\]
Applying the union bound we obtain
\begin{align*}
    \Prob[R_2(C)>\rho n]&=\Prob\sparenv*{\bigcup_{\bv\in G_q^{2\times n}}\set*{X_{\bv}=0}} \leq \sum_{\bv\in G_q^{2\times n}}\Prob\sparenv*{X_{\bv}=0} \\
    &\overset{(a)}{\leq} q^{2n}\exp\parenv*{-q^{n(2\varepsilon+o(1))}}=\exp\parenv*{-q^{n(2\varepsilon+o(1))}},
\end{align*}
where (a) follows by the fact that the $o(1)$ term in the upper-bound on $\Prob[X_{\bv}=0]$ does not depend on the choice of $\bv$, as described in Proposition~\ref{prop:ProbBoundWord}.

For a bound on the probability of the second event, we use~\eqref{eq:BinBound}, and the fact that the cardinality of our random code $C$ has a $\mathrm{Bin}(q^n,q^{-n(H_{q^2}(\rho) -\varepsilon)})$ distribution:
\begin{align*}
    \Prob\sparenv*{\abs{C}\geq  q^{n(1-H_{q^2}(\rho) +\varepsilon+1/n)}}&\leq\Prob\sparenv*{\abs{C}\geq 2 \E\sparenv*{\abs*{C}}}\\
    &\leq \exp\parenv*{-\frac{1}{3}q^{n(1-H_{q^2}(\rho)+\varepsilon)}} \\
    &=\exp\parenv*{-q^{n(1-H_{q^2}(\rho)+\varepsilon+o(1))}}.
\end{align*}
Hence,
\begin{align*}
    &\Prob\sparenv*{\set*{R_2(C)\leq \rho n}\cap \set*{\abs{C}< q^{n(1-H_{q^2}(\rho) +\varepsilon+1/n)}}}\\
    &\quad =1-\Prob\sparenv*{\set*{R_2(C)> \rho n}\cup \set*{\abs{C}\geq  q^{n(1-H_{q^2}(\rho) +\varepsilon+1/n)}}}\\
    &\quad \geq 1-\Prob\sparenv*{R_2(C)> \rho n}-\Prob\sparenv*{ \abs{C}\geq  q^{n(1-H_{q^2}(\rho) +\varepsilon+1/n)}}\\
    &\quad \geq  1
    - \exp\parenv*{-q^{n(2\varepsilon+o(1))}}
    - \exp\parenv*{-q^{n(1-H_{q^2}(\rho)+\varepsilon+o(1))}} \\
    &\quad \geq 1-\exp\parenv*{-q^{n(2\varepsilon+o(1))}},
\end{align*}
where the last inequality follows, for all sufficiently large $n$, from the fact that $\varepsilon < H_q\parenv{\frac{q}{q+1}\rho}-H_{q^2}(\rho) \leq 1-H_{q^2}(\rho)$.
\qed
\end{proof}

The proof of Theorem~\ref{th:BinRatecovering} now easily follows.

\begin{proof}[Theorem~\ref{th:BinRatecovering}]
As explained in the beginning of Section~\ref{sec:SecondOreder}, the only remaining part we need to prove is that $\kappa_2(\rho,q)\leq 1-H_{q^2}(\rho)$ for all $\rho\in(0,1-\frac{1}{q^2})$. Let  $\rho\in(0,1-\frac{1}{q^2}))$ and $\varepsilon\in\parenv{0,H_{q}\parenv{\frac{q}{q+1}\rho}-H_{q^2}(\rho)}$ be fixed. By Proposition~\ref{prop:RandomCodingGood}, for all sufficiently large $n$, the random code $C$ satisfies
\[ \Prob\sparenv*{\set*{R_2(C)\leq \rho n}\cap \set*{\abs{C'}< q^{n(1-H_{q^2}(\rho) +\varepsilon+1/n)}}}>0.\]
In particular, there exists at least one (deterministic) code $C_n$ such that 
\[ R_2(C_n)\leq \rho n,\quad  \text{and}\quad \abs{C_n}< q^{n(1-H_{q^2}(\rho) +\varepsilon+1/n)}. \]
This immediately implies that,
\[ k_2(n,\rho n,q)\leq q^{n(1-H_{q^2}(\rho) +\varepsilon+1/n)}=q^{n(1-H_{q^2}(\rho) +\varepsilon+o(1))}. \]
This proves that
\begin{align*}
    \kappa_2(\rho,q)&=\liminf_{n\to \infty}\frac{1}{n}\log_q( k_2(n,\rho n,q))\\
    &\leq \liminf_{n\to \infty} \parenv*{1-H_{q^2}(\rho) +\varepsilon+o(1)}=1-H_{q^2}(\rho) +\varepsilon.
\end{align*}
Taking $\varepsilon\to 0$ we conclude 
\[\kappa_2(\rho,q)\leq 1-H_{q^2}(\rho).\]
\qed
\end{proof}

While for the proof of the upper-bound part of Theorem~\ref{th:BinRatecovering} it is only required to show the existence of second-order covering codes, we observe that a stronger conclusion may follow, namely, that with high probability, a random code generated according to our distribution is a second-order covering code. We use this fact in order to prove that the fraction of second-order covering codes (among the set of codes of sufficiently large size) tends to $1$ as $n\to\infty$.

For an integer $q\geq 2$, $\rho\in(0,1-\frac{1}{q^2})$, $n\in \N$, and $0\leq M \leq q^n$, let  $\alpha_q(n,\rho,M)$ denote the fraction of codes of length $n$ over $G_q$ with second covering radius at most $\rho n$ in the set of $\parenv*{n,M}_q$ codes. Namely,
\[\alpha_q(n,\rho,M)\eqdef\frac{\abs*{\cC_q(n,\rho,M)}}{\binom{q^n}{M}},\]
where 
\[ \cC_q(n,\rho,M)\eqdef \set*{C\subseteq  G_q^n; R_2(C)\leq \rho n, \abs*{C}=M}.\]

\begin{lemma}\label{lem:FracMonotonic}
Let $\rho$ and $n$ be fixed. For any $0\leq M\leq q^n-1$,
\[\alpha_q(n,\rho,M)\leq \alpha_q(n,\rho,M+1). \]
\end{lemma}
\begin{proof}
We start by observing that any $(n,M+1)_q$ code $C$ that contains a sub-code in $\cC_q(n,\rho,M)$  satisfies $C\in\cC_q(n,\rho,M+1) $. This immediately gives a lower bound on $|\cC_q(n,\rho,M+1)|$
\begin{equation}\label{eq:lemFraction1}
    |\cC_q(n,\rho,M+1)|\geq  \abs*{\set*{C\subseteq  G_q^{n} ; \exists C'\subseteq C, C'\in \cC_q(n,\rho,M) }}. 
\end{equation}
We note that any code in $\cC_q(n,\rho,M)$ can be extended to an $(n,M+1)_q$ code by adding one of the $q^n-M$ remaining vectors in $ G_q^n$. We also note that any code obtained by adding a vector to a code in $\cC_q(n,\rho,M)$ has $M+1$ subcodes of size $M$. Therefore, any such code may be the extension of at most $M+1$ codes in $\cC_q(n,\rho,M)$, which implies that 
 \begin{equation} \label{eq:lemFraction2}
     \abs*{\set*{C\subseteq  G_q^{n} ; \exists C'\subseteq C, C'\in \cC_q(n,\rho,M) }}\geq\abs{\cC_q(n,\rho,M)}\frac{q^n-M}{M+1}.
 \end{equation}
Using~\eqref{eq:lemFraction1} and~\eqref{eq:lemFraction2} we conclude that
\[
\frac{\alpha_q(n,\rho,M+1)}{\alpha_q(n,\rho,M)}=\frac{\abs*{\cC_q(n,\rho,M+1)}}{\binom{q^n}{M+1}}\cdot \frac{\binom{q^n}{M}}{\abs*{\cC_q(n,\rho,M)}}
\geq \frac{q^n-M}{M+1}\cdot\frac{\binom{q^n}{M}}{\binom{q^n}{M+1}}=1.
\]
\qed
\end{proof}

The following theorem asserts that the fraction of second-order covering codes tends to $1$ as the length tends to infinity when we consider codes with rate larger then the optimal rate presented in Theorem~\ref{th:BinRatecovering} by an arbitrarily small amount. 

\begin{theorem} \label{th:FractionCovering}
For any $\rho\in(0,1-\frac{1}{q^2})$ and $\varepsilon>0$, let us denote $M(n,\rho,\varepsilon)\eqdef\floor{q^{n(1-H_{q^2}(\rho)+\varepsilon+1/n)}}$. Then,
\[\lim_{n\to\infty}\alpha_q(n,\rho,M(n,\rho,\varepsilon))=1.\]
\end{theorem}
\begin{proof}
We start by observing that by the monotonicity property given in Lemma~\ref{lem:FracMonotonic}, it is sufficient to prove the claim for any $\varepsilon\in \parenv{0,H_2\parenv{\frac{q}{q+1}\rho}-H_{q^2}(\rho)}$. Thus, let as assume that $\varepsilon\in \parenv{0,H_2\parenv{\frac{q}{q+1}\rho}-H_{q^2}(\rho)}$ and let $C$ be the random code as in Proposition~\ref{prop:RandomCodingGood}, where it is shown that for all sufficiently large $n$,
\[\Prob\sparenv*{\set*{R_2(C)\leq \rho n}\cap \set*{\abs{C}< q^{n(1-H_{q^2}(\rho) +\varepsilon+1/n)}}}\geq 1-\exp\parenv*{-q^{n(2\varepsilon+o(1))}}.\]
We note that 
\begin{align*}
    &\Prob\sparenv*{\set*{R_2(C)\leq \rho n}\cap \set*{\abs{C}< q^{n(1-H_{q^2}(\rho) +\varepsilon+1/n)}}}\\
    &\qquad \leq \Prob\sparenv*{\set*{R_2(C)\leq \rho n}\cap \set*{\abs{C}\leq M(n,\rho,\varepsilon)}}\\
    &\qquad = \sum_{m=0}^{M(n,\rho,\varepsilon)}\Prob\sparenv*{\set*{R_2(C)\leq \rho n}\cap \set*{\abs{C}=m}}\\
    &\qquad =  \sum_{m=0}^{M(n,\rho,\varepsilon)}\Prob\sparenv*{R_2(C)\leq \rho n \bigg| \abs{C}=m}\Prob\sparenv*{\abs{C}=m}.
\end{align*}

We recall that $C$ is generated randomly by picking any vector in $G_q^{n}$ with the same probability, independently of the remaining vectors in $ G_q^n$. This implies that the probability of any code is determined by its cardinality. Thus, under the conditional measure on the event $\set{\abs{C}=m}$, $C$ is uniformly distributed on the set of $(n,m)_q$ codes, and in particular 
\[ \Prob\sparenv*{R_2(C)\leq \rho n \bigg| \abs{C}=m}= \alpha_q(n,\rho,m).\]

Combining the above results with the monotonicity property from Lemma~\ref{lem:FracMonotonic} we have, for all sufficiently large $n$,
\begin{align*}
    1-\exp\parenv*{-q^{n(2\varepsilon+o(1))}} &\leq  \sum_{m=0}^{M(n,\rho,\varepsilon)}\Prob\sparenv*{R_2(C)\leq \rho n \bigg| \abs{C}=m}\Prob\sparenv*{\abs{C}=m} 
    \\ &=\sum_{m=0}^{M(n,\rho,\varepsilon)}\alpha_q(n,\rho,m)\Prob\sparenv*{\abs{C}=m} \\
    &\leq \alpha_q(n,\rho,M(n,\rho,\varepsilon)) \sum_{m=0}^{M(n,\rho,\varepsilon)}\Prob\sparenv*{\abs{C}=m}\\
    &= \alpha_q(n,\rho,M(n,\rho,\varepsilon))\Prob\sparenv*{\abs{C}\leq M(n,\rho,\varepsilon)}\\
    &\leq \alpha_q(n,\rho,M(n,\rho,\varepsilon)).
\end{align*}
This completes the proof.
\qed
\end{proof}

\section{Conclusion and Further Questions}
\label{sec:conclude}

In this paper we studied the optimal rate of general second-order covering codes over finite Abelian groups. Our main result, the exact asymptotic minimal rate of second-order covering codes, was proved using a probabilistic approach. 

As we saw, the problem of finding a $t$-th-order covering code over a group $G_q$ is in fact the problem of covering the space of matrices $G_q^{t\times n}$ (equipped with the $t$-metric) with codes that are $t$-powers, i.e., codes of the form $C^t$ for some one-dimensional code $C\subseteq G_q^n$. It is easy to check that $G_q^{t\times n}$ (with the $t$-metric) is isometrically isomorphic (as a metric space) to $(G_q^t)^n$ with the regular Hamming metric.  Combining this observation with \eqref{eq:1orderSol}, we conclude that the asymptotic minimal rate for general first-order covering codes in $G_q^{t\times n}$ with normalized covering radius $\rho$ with respect to the $t$-metric is 
\[\kappa_1(\rho,q^t)=\begin{cases} 1-H_{q^t}(\rho) & \rho \in [0,1-\frac{1}{q^t}),\\
0 & \rho \in [1-\frac{1}{q^t},1].
\end{cases}\]
In the second-order case, Theorem~\ref{th:BinRatecovering} reveals that 
\[ \kappa_2(\rho,q)=\kappa_1(\rho,q^2)=\begin{cases} 1-H_{q^2}(\rho) & \rho \in [0,1-\frac{1}{q^2}),\\
0 & \rho \in [1-\frac{1}{q^2},1].
\end{cases} \] 
In theory, the proof strategy of Theorem~\ref{th:BinRatecovering} may be applied for higher values of $t$. However in practice,  the analysis performed in our proof, which is already involved in the second-order case, seem not to be scalable for higher orders. We therefore leave the higher-order problem for future study:
\begin{question}
Prove that for any $q,t\geq 2$
\[\kappa_t(\rho,q)=\kappa_1(\rho,q^t)=\begin{cases} 1-H_{q^t}(\rho) & \rho \in [0,1-\frac{1}{q^t}),\\
0 & \rho \in [1-\frac{1}{q^t},1]. \end{cases}\]
\end{question}

Another interesting direction of research involves linear codes. In the case where $G_q=\F_q$ is a finite field, \eqref{eq:1orderSol} raises the suspicion that there is no different in asymptotic minimal rate between general codes and linear codes. Thus, we suggest the following open problem as well:
\begin{question}
Prove or disprove that for any $q,t\geq 2$
\begin{equation}
    \label{eq:Q2}
    \kappa_t(\rho,q)=\kappa_t^{\mathrm{Lin}}(\rho,q).
\end{equation}
\end{question}
Except for the first-order case, it is unknown whether~\eqref{eq:Q2} is true.

\bibliographystyle{elsarticle-num}

\end{document}